\documentclass[11pt]{article}
\usepackage[latin1]{inputenc}
\usepackage[T1]{fontenc}
\usepackage[frenchb,english]{babel}
\usepackage[mathscr]{eucal}
\usepackage{amsbsy}
\usepackage{amssymb}
\usepackage{amsmath}
\usepackage{amsthm}
\usepackage{bm}
\usepackage{color}
\newcommand{\cd}{\mathcal{D}}
\newcommand{\cf}{\mathcal{F}}
\newcommand{\cF}{\mathcal{F}}
\newcommand{\cB}{\mathcal{B}}
\newcommand{\cA}{\mathcal{A}}
\newcommand{\cC}{\mathcal{C}}
\newcommand{\cW}{\mathcal{W}}
\newcommand{\cQ}{\mathcal{Q}}

\newcommand{\cS}{\mathcal{S}}
\newcommand{\cL}{\mathcal{L}}

\newcommand{\eZ}{\mathbb Z}
\newcommand{\eR}{\mathbb R}
\newcommand{\eN}{\mathbb N}
\newcommand{\e}{\varepsilon}

\theoremstyle{plain}
\newtheorem{theorem}{Theorem}
\newtheorem{lemma}[theorem]{Lemma}
\newtheorem{corollary}[theorem]{Corollary}
\newtheorem{proposition}[theorem]{Proposition}
\theoremstyle{definition}
\newtheorem{definition}[theorem]{Definition}

\newtheorem{remark}[theorem]{Remark}

\newcommand{\FF}{\mathbb F}
\newcommand{\RR}{\mathbb R}
\newcommand{\NN}{\mathbb N}

\newcommand{\ZZ}{\mathbb Z}
\newcommand{\beq}{\begin{equation}}
\newcommand{\eeq}{\end{equation}}
\newcommand{\wc}{\xrightarrow{\mathscr{D}}}
\newcommand{\wcs}{\xrightarrow{\mathscr{D}^*}}
\newcommand{\wsc}{\stackrel{\cd}{\Longrightarrow}}

\newcommand{\pc}{\xrightarrow[n\to\infty]{\mathrm{P}}}

\newcommand{\CLA}{{\cal A}}

\newcommand{\CLF}{{\cal F}}
\newcommand{\id}{{1\!\!\mathbb I}}
\newcommand{\wh}{\widehat}
\newcommand{\wt}{\widetilde}

\newcommand{\PP}{\mathsf P}
\newcommand{\EE}{\mathsf E}

\newcommand{\eE}{\mathbb E}

\newcommand{\ip}[1]{\langle #1 \rangle}
\newcommand{\cT}{\mathcal T}
\newcommand{\diag}{\mathrm{diag}}

\definecolor{zalia}{rgb}{0.0,0.4,0.0}
\definecolor{raudona}{rgb}{0.8,0.0,0.0}

\title{Uniform asymptotic normality of
weighted sums of short-memory linear 
processes\footnote{The research supported  by the Research Council of Lithuania, grant No. S-MIP-17-76}}

\author{Rimas Norvai\v sa and Alfredas Ra\v ckauskas\\
Vilnius university, Institute of applied mathematics}

\begin{document}

\maketitle

\begin{abstract}
Let $X_1, X_2,\dots$ be a short-memory linear
process of random variables.
For $1\leq q<2$, let $\cF$ be a bounded set
of real-valued functions on $[0,1]$ with
finite $q$-variation.
It is proved that $\{n^{-1/2}\sum_{i=1}^nX_if(i/n)\colon\,f\in\cF\}$ converges in outer
distribution in the Banach space of bounded functions on $\cF$ as $n\to\infty$.
Several applications to a regression model and a multiple change point model
are given.
\end{abstract}

\section{Introduction}

Let $\ZZ$ and $\NN$ be the sets of all integers and nonnegative integers, respectively.
Let $(\eta_j)_{j\in\ZZ}$ be a sequence of independent identically distributed  
random variables on a probability space $(\Omega, \cA, \PP)$ with mean zero and
finite second moment $\sigma^2_\eta=\EE\eta_1^2\not =0$. 
A sequence $(X_i)_{i\in\ZZ}$ of random
variables  defined by
\begin{equation}\label{lin-process}
X_{i}=\sum_{j=0}^\infty \psi_j\eta_{i-j},
\quad\,\, i\in\ZZ,
\end{equation}
is a \emph{linear process} provided a sequence of real numbers 
$(\psi_j)_{j\in\NN}$ is square summable.
We say that $(X_i)_{i\in\ZZ}$ and its subsequence $X_1,X_2,\dots$ are 
\emph{short-memory linear processes} (with 
innovations $(\eta_j)$ and summable filter
$(\psi_j)$) if, in addition,
\begin{equation}\label{filter:1}
 \sum_{j=0}^\infty |\psi_j|<\infty\ \ 
 \textrm{and}\ \ A_{\psi}:=\sum_{j=0}^\infty 
 \psi_j\not=0.
\end{equation}

Given a short memory linear process $X_1, X_2,\dots$, a function 
$f\colon\,[0,1]\to\RR$ and a positive integer $n\in\NN_{+}$, 
let $\nu_n(f)$ be the \emph{$n$-th $f$-weighted sum of linear process}
 defined by 
\begin{equation}\label{nun}
\nu_n(f):=\sum_{i=1}^n X_if
\Big (\frac{i}{n}\Big ).
\end{equation}
If $\cF$ is a class of real-valued measurable
functions on $[0,1]$, then $\nu_n=
\{\nu_n(f)\colon\,f\in\cF\}$ is the
\emph{$n$th $\cF$-weighted sum of 
linear process}, and $\nu_n$, 
$n\in\NN_{+}$, is a sequence of
 \emph{weighted sums of linear process}.
This type of weighting of random variables has a number of applications
in statistics and econometrics (see \cite{ADGK} and references therein).
Some new applications are suggested in the last section of the present paper.

In this paper we consider classes of functions of bounded $q$-variation
with $q\in [1,2)$.
Given a class of functions $\cF$, let $\ell^{\infty}(\cF)$ be the Banach space
of all uniformly bounded real-valued functions $\mu$ on $\cF$ endowed with the 
uniform norm  
\begin{equation}\label{Fsup}
\|\mu\|_{\cF} := \sup\{|\mu(f)|: f \in \cF\}.
\end{equation}
Each $n$th $\cF$-weighted sum of linear process $\nu_n$ has sample paths in 
$\ell^{\infty}(\cF)$.
Whenever $\cF$ is infinite set the Banach space $\ell^\infty(\cF)$ is non-separable. 
We show that normalized sequence of weighted sums of linear process converge in 
$\ell^{\infty}(\cF)$ in outer distribution as defined by
J. Hoffmann-J{\o}rgensen (Definition \ref{HJ} below).
Next is the main result of the paper.

\begin{theorem}\label{teo2a} 
Let $X_1, X_2,\dots$ be a short-memory linear
process given by {\rm(\ref{lin-process})}, 
let $1 \le  q < 2$ and let
$\cF$ be a bounded set of functions on $[0,1]$ with bounded
$q$-variation.
There exists a version of the 
isonormal Gaussian process $\nu$
restricted to $\cF$ with values in a
separable subset of $\ell^{\infty}(\cF)$,
it is measurable for the Borel sets on its
range and 
\begin{equation}\label{1teo2}
n^{-1/2}\nu_n \wcs \sigma_{\eta}A_{\psi}\nu
\qquad\mbox{in $\ell^\infty(\cF)$ as
$n\to\infty$,}
\end{equation}
where $\sigma_{\eta}$ and $A_{\psi}$ are
parameters describing the short-memory
linear process.
\end{theorem}

A weak invariance principle for sample paths of partial sum process based on a short memory linear process (Theorem \ref{teo1}) is obtained using the preceding theorem 
and a duality type result (Theorem \ref{dualitymap}).

The paper is organized as follows. 
Section \ref{Notation} contains notation and further results. 
Section \ref{Marginals} contains the proof of convergence of finite dimensional
distributions of the process $n^{-1/2}\nu_n$. 
Asymptotic equicontinuity is discussed in 
Section \ref{Equicontinuity}.
The proof of the main result, Theorem \ref{teo2a}, is given in Section \ref{MainResults}.
Further results, Theorems  \ref{teo1} and \ref{dualitymap}, are proved in Section
\ref{Pvariation}.
Finally, Section \ref{Applications} is devoted to some applications.

\section{Notation and results}\label{Notation}

Since processes considered in this paper have sample paths in non-separable Banach spaces
we use the concept of convergence in outer distribution developed  
by J. Hoffmann-J{\o}rgensen.
Given a probability space $(\Omega.\cA,P)$, let $T$ be a function from $\Omega$ to
the extended real line $\bar{\RR}$.
The outer integral of $T$ with respect to $P$ is defined as
$$
E^{\ast}T:=\inf\big\{EU\colon\,\mbox{$U\colon\,\Omega\to\bar{\RR}$ is measurable,
$EU$ exists and $U\geq T$}\big\}. 
$$
The outer probability of an arbitrary subset $B$ of $\Omega$ is
$P^{\ast}(B):=E^{\ast}\id_B=\inf\{P(A)\colon\,A\in\cA,\,\,A\supset B\}$,
here and elsewhere $\id_B$ is the indicator function of a set $B$.
  
\begin{definition}\label{HJ} 
Let $\eE$ be a metric space. 
For each $n \in\eN$, let 
$(\Omega_n, \cA_n, P_n)$ be a probability
space and let $Z_n$ be a function from 
$\Omega_n$ into $\eE$. 
Suppose that $Z_0$ takes values in some
 separable subset of $\eE$ and is measurable 
 for the Borel sets on its range. 
It is said that the sequence $(Z_n)$ converges
in outer distribution to $Z_0$, denoted 
$Z_n \wcs Z_0$,
if, for every bounded continuous function $h:\eE \to \eR$,
$$
\lim_{n\to\infty}E^*h(Z_n)=Eh(Z_0).
$$
\end{definition}

\begin{remark} If $Z_n, n = 0, 1, \dots$ are random elements taking values in a separable metric space
$\eE$ endowed with the Borel $\sigma$-algebra, then the convergence $Z_n \wcs  Z_0$ is equivalent to usual
convergence in distribution $Z_n\wc Z_0$:
$$
\lim_{n\to\infty}Eh(Z_n)=Eh(Z_0)
$$
for every bounded continuous function $h: \eE \to \eR$.
\end{remark}

To establish convergence in outer distributions
on $\ell^{\infty}(\cF)$ we need a separable
subset for a support of a limit distribution.
Let $UC(\cF,d)$ be a set of all 
$\nu\in \ell^\infty(\cF)$ which are
uniformly $d$-continuous.
The set $UC(\cF, d )$ is 
separable subspace of $\ell^{\infty}(\cF)$
if and only if $(\cF,d)$ is totally bounded.
As usual $N(\e, \cF, d)$ is the minimal number
 of open balls of $d$-radius $\e$ which are 
 necessary to cover $\cF$.
The pseudometric space $(\cF,d)$ is totally
bounded if $N(\e, \cF, d)$ is finite for every 
$\e > 0$. 
This  property always holds under the
 assumptions imposed below.

Let $\cL^2[0,1]=\cL^2([0,1],\lambda)$ be a
set of measurable functions which are 
square-integrable for Lebesgue measure $\lambda$ on $[0,1]$  with a pseudometric 
$\rho_2(f, g)=\rho_{2,\lambda}(f, g)=(\smallint_{[0,1]}(f-g)^2\,d\lambda)^{1/2}$. 
Let $L^2[0,1]=L^2([0,1],\lambda)$ be the
associated Hilbert space endowed with the inner product 
$\ip{f, g}=\int_0^1 f(t)g(t)\lambda(dt)$.
Given a set $\cF\subset \cL^2[0, 1]$, let $\nu=\{\nu(f)\colon\,f\in \cF\}$ be
a centred Gaussian process such that $\EE[\nu(f)\nu(g)]=\ip{f, g}$ for all 
$f, g\in \cF$. 
Such process exists and provides a linear isometry from $L^2[0, 1]$ to 
$L^2(\Omega, \cF, \PP)$.   
By Dudley \cite{Dudley2} or 
\cite[2.6.1 and 2.8.6 Theorems]{Dudley5}, if
\begin{equation}\label{ME}
\int_0^1\sqrt{\log N(x,\cF,\rho_2)}\,dx
<\infty
\end{equation}
then $\nu=\{\nu(f)\colon\,f\in\cF\}$ admits
a version with almost all sample paths
bounded and uniformly continuous on $\cF$
with respect to $\rho_2$.
In what follows we denote a suitable version 
by the same notation $\nu$, and so $\nu$ itself
takes values in 
$UC(\cF, \rho_2)$ and is measurable for the Borel sets on its range. 

In this paper the condition (\ref{ME}) is applied to sets $\cF$ defined as follows.
For $-\infty <a<b<\infty$ and
$0 < p < \infty$ the $p$-variation of a function $g\colon\,[a,b]\to\RR$ is the
supremum 
$$
v_p(g;[a, b]):=\sup\Big\{\sum^m_{i=1}
|g(t_i) - g(t_{i-1})|^p: a=t_0 < t_1 < \cdots< t_m = b, m \in\NN_{+}\Big\},
$$
which can be finite or infinite.
If $v_p(g;[a, b]) < \infty$ then $g$ is said to have bounded $p$-variation and the set of all such functions is denoted by 
$\cW_p[a, b]$. 
We abbreviate $v_p(g):= v_p(g;[0, 1])$.
For each $g \in \cW_p[0, 1]$ and 
$1 \le p < \infty$, let 
$\|g\|_{(p)} := v^{1/p}_p(g)$. Then 
$\|g\|_{(p)}$ is a seminorm equal to zero 
only for constant functions $g$.
The $p$-variation norm is 
$$
\|g\|_{[p]} :=\|g\|_{\sup} +\|g\|_{(p)}
$$
where $\|g\|_{\sup}:=\sup_{0\le t\le 1}|g(t)|$.
The set $\cW_p[0, 1]$ is a non-separable 
Banach space with the norm $\|\cdot\|_{[p]}$. 
If $\cF$ is a bounded subset of $\cW_q[0,1]$ with $1\leq q<2$,
then (\ref{ME}) holds by the proof of Theorem 2.1 in \cite{Dudley4} 
(see also \cite[Theorem 5]{NandR08}).

Now we are prepared to formulate further results.
Let $X_1, X_2,\dots$ be a sequence of real-valued random variables.
For each positive integer $n\in\NN_{+}$, the \emph{$n$th partial sum process} of random
 variables is defined by 
$$
S_n(t):=\sum_{i=1}^{\lfloor nt\rfloor}X_i
=\sum_{i=1}^nX_i\id_{[0,t]}\Big (
\frac{i}{n}\Big ),\quad t\in [0,1].
$$
Here for a real number $x\geq 0$, $\lfloor x\rfloor:=\max\{k\colon\, k\in\NN,\,k\leq x\}$ is a value of the floor function.
Then the \emph{partial sum process} is the sequence of $n$th partial sum processes 
$S_n=\{S_n(t)\colon\,t\in [0,1]\}$,  $n\in\NN_{+}$.
Let $W$ be a Wiener process on $[0,1]$. 
In \cite{NandR08}, assuming that random variables  $X_1, X_2,\dots$ are independent
and identically distributed, it is proved that convergence in outer distribution
\begin{equation}\label{2008result}
n^{-1/2}S_n\wcs \sigma W\quad
\mbox{in $\cW_p[0,1]$ as $n\to\infty$,}
\end{equation}
holds if and only if $\EE X_1=0$ and
$\sigma^2=\EE X_1^2<\infty$.
The assumption $p>2$ can't be replaced by
$p=2$ since in this case the limiting
process $W$ does not belong to $\cW_2[0,1]$.
The next theorem extends this fact to the case where a sequence of random variables
$X_1, X_2,\dots$ is a short-memory linear process.

\begin{theorem}\label{teo1} 
Let $X_1, X_2, \dots$ be a short-memory
linear process, let $p>2$ and let $W$ be a Wiener
process on $[0,1]$. 
Then
\begin{equation}\label{teo1:result1}
n^{-1/2}S_n \wcs\sigma_{\eta} A_{\psi}
W\quad\mbox{in $\cW_p[0, 1]$ as 
$n\to\infty$.}
\end{equation}
\end{theorem}

For any $p>0$,
the $p$-variation of a sample function of
the $n$th partial sum is
\begin{equation}\label{teo1b}
v_p(S_n)=\max\Big\{\sum_{j=1}^m\Big |
\sum_{i=k_{j-1}+1}^{k_j}X_i\Big |^p\colon\,
0=k_0<\cdots <k_m=n,\,\,1\leq m\leq n\Big\}.
\end{equation}
Theorem \ref{teo1} and
continuous mapping theorem (e.g. Theorem
1.3.6 in \cite{VandW96}) applied to the
$p$-variation yield the following.

\begin{corollary}\label{teo1a}
Under the hypotheses of Theorem 
{\rm \ref{teo1}}, we have
$$
n^{-\frac{p}{2}}v_p(S_n)\wc \sigma_{\eta}^p
 A_{\psi}^pv_p(W)\quad\mbox{as  $n\to\infty$.}
$$
\end{corollary} 

Suppose that $\cF$ contains the family of
indicator functions of subintervals of 
$[0,1]$.
Then the $n$th partial sum process
of a linear process $S_n$ and the $n$th 
$\cF$-weighted sum of linear process
$\nu_n$ are related by the equality
\begin{equation}\label{Snun}
S_n(t)=\nu_n(\id_{[0,t]})\quad\mbox{ for
each $t\in [0,1]$.}
\end{equation}
This relation is used in the following theorem
to obtain  Theorem \ref{teo1} from a uniform
convergence of $n^{-1/2}\nu_n$ over the set 
$\cF_q=\{f\in\cW_q[a,b]\colon\, \|f\|_{[q]}\leq 1\}$, $1\leq q<2$,
which is the unit ball in  $\cW_q[a,b]$.
For this aim the $n$-th $\cF_q$-weighted
sum of linear process $\nu_n$ is considered
as a bounded linear functional over
$\cW_q[a,b]$.

\begin{theorem}\label{dualitymap}
Let $1<p<\infty$ and $1<q<\infty$ be such
that $p^{-1}+q^{-1}=1$.
For a linear bounded functional
$L\colon\,\cW_q[a,b]\to\RR$  let
$T(L)(t):=L(\id_{[a,t]})$ for each 
$t\in [a,b]$.
Then $T$ is a linear mapping from the
dual space $\cW_q^{\ast}[a,b]$ into 
$\cW_p[a,b]$ and 
\begin{equation}\label{1dualitymap}
\|T(L)\|_{[p]}\leq 5\|L\|_{\cF_q},\quad
L\in\cW_q^{\ast}[a,b],
\end{equation}
where $\|\cdot\|_{\cF_q}$ is defined by {\rm (\ref{Fsup})}.
\end{theorem}

To prove Theorem \ref{teo2a} we use the
asymptotic equicontinuity criterion for
convergence in law in $\ell^{\infty}(\cF)$
(see e.g. \cite[Theorem 3.7.23]{GandN16}
or \cite[p. 41]{VandW96}).
The conclusion of Theorem \ref{teo2a}
holds if and only if $(i)$, $(ii)$ and $(iii)$
hold, where
\begin{itemize}
\item[$(i)$] the finite dimensional distributions of 
$n^{-1/2}\nu_n$ converge in distribution to the corresponding 
finite dimensional distributions of $\nu$;
\item[$(ii)$] $n^{-1/2}\nu_n$ is
 asymptotically equicontinuous with respect
  to $\rho_2$;
\item[$(iii)$] $\cF$ is totally bounded
for $\rho_2$.
\end{itemize}

\section{Convergence of finite dimensional distributions}\label{Marginals}

In this section we establish convergence
of finite dimensional distributions of the
processes $n^{-1/2}\nu_n$.
Recall that $\cF_q=\{f\in\cW_q[0, 1]
\colon\,\|f\|_{[q]}\le 1\}$ is endowed with
the pseudometric $\rho_2$.
We begin with a one-dimensional case.

We do not know results in the literature
which yield the convergence in distribution of
real random variables $n^{-1/2}\nu_n(g)$
when $g\in\cW_q[0,1]$ for some $q\in [1,2)$
under the hypotheses of Theorem \ref{CM1}
below.
The best available related results are
due to K.M. Abadir et all \cite{ADGK}
which give the desired fact when $g$ has
bounded total variation.
Next is a more general result for 
short-memory linear process with
independent identically distributed
inovations and weights given by a 
function $g$.

\begin{theorem}\label{CM1}
Suppose $(X_i)_{i\in\ZZ}$ is a linear 
process defined by {\rm (\ref{lin-process})} and {\rm (\ref{filter:1})}, and $\nu$ is
the isonormal Gaussian processes 
on $\cL_2[0,1]$.
If $g\in\cW_q[0,1]$ for some
$1\leq q<2$, then
\begin{equation}\label{1CM1}
n^{-1/2}\nu_n(g)\wc \sigma_{\eta}A_{\psi}
\nu (g),\quad\mbox{as $n\to\infty$.}
\end{equation}
\end{theorem}

\begin{proof}
Let $1\leq q<2$ and $g\in\cW_q[0,1]$.
For each $n\in\NN_{+}$ and $k\in\NN$, let
\begin{equation}\label{6CM1}
T_{nk}:=\sum_{i=1}^n\eta_{i-k}g\Big (
\frac{i}{n}\Big ).
\end{equation}
By (\ref{nun}) and (\ref{lin-process}) we
have the representation
\begin{eqnarray*}
\nu_n(g)&=&\sum_{i=1}^n\Big (
\sum_{k=0}^{\infty}\psi_k\eta_{i-k}\Big )
g\Big (\frac{i}{n}\Big )
=\sum_{k=0}^{\infty}\psi_kT_{nk}\\
&=&\sum_{k=0}^{\infty}\psi_k\big [T_{nk}
-T_{n0}\big ]+ A_{\psi}T_{n0}.
\end{eqnarray*}
Since function $g\in\cW_q[0,1]$, it is regulated (see e.g. \cite[p. 213]{DN99}). 
Thus $g^2$ is Riemann integrable, and so
$$
Var\Big (n^{-1/2}T_{n0}\Big )=
\frac{\sigma_{\eta}^2}{n}\sum_{i=1}^n
g^2\Big (\frac{i}{n}\Big )\to 
\sigma_{\eta}^2\int_0^1g^2\,d\lambda,\quad
\mbox{as $n\to\infty$.}
$$
Since $\nu$ is the isonormal Gaussian
processes on $\cL_2[0,1]$ it follows
by the Lindeberg central limit theorem that
$$
n^{-1/2}T_{n0}=\frac{1}{\sqrt{n}}
\sum_{i=1}^n\eta_{i}g\Big (\frac{i}{n}\Big )
\wc \sigma_{\eta}\nu (g),\quad\mbox{as
$n\to\infty$.}
$$
Therefore to prove (\ref{1CM1}),
due to Slutsky theorem, it is enough to
show that
\begin{equation}\label{5CM1}
R_n:=
\sum_{k=0}^{\infty}\frac{\psi_k}{\sqrt{n}}
\big [T_{nk}-T_{n0}\big ]\to 0
\quad\mbox{in probability $\PP$ as
$n\to\infty$.} 
\end{equation}
We will show that the following two
properties hold true:
\begin{equation}\label{2CM1}
\sup_{n,k}\frac{1}{n}\EE T_{nk}^2<\infty
\end{equation}
and
\begin{equation}\label{3CM1}
\mbox{for each $k\in\NN$}\quad
\lim_{n\to\infty}
\frac{1}{\sqrt{n}}|T_{nk}-T_{n0}|= 0
\quad\mbox{in probability $\PP$.} 
\end{equation}
For the moment
suppose that (\ref{2CM1}) and (\ref{3CM1})
hold true.
Let $\epsilon >0$ and $K\in\NN$.
Split the sum $R_n$ given by (\ref{5CM1}) into the sum with all $k\leq K$ and the
sum with all $k>K$ to get the 
inequality
\begin{eqnarray}
\lefteqn{\PP(\{|R_n|>\epsilon\})}\nonumber\\
&\leq&\PP\Big (\Big \{ \sum_{k=0}^K\frac{|\psi_k|}{
\sqrt{n}}|T_{nk}-T_{n0}|>\frac{\epsilon}{2}
\Big\}\Big )+
\PP\Big (\Big\{\sum_{k>K}\frac{|\psi_k|}{
\sqrt{n}}|T_{nk}-T_{n0}|>\frac{\epsilon}{2}
\Big\}\Big ).\label{4CM1}
\end{eqnarray}
Clearly we have the bound 
$$
\PP\Big (\Big\{\sum_{k>K}\frac{|\psi_k|}{
\sqrt{n}}|T_{nk}-T_{n0}|>\frac{\epsilon}{2}
\Big\}\Big )\leq\frac{4}{\epsilon}\sup_{n,k}
\Big (\frac{\EE T_{nk}^2}{n}\Big )^{1/2}
\sum_{k>K}|\psi_k|.
$$
By (\ref{2CM1}) and (\ref{filter:1}), 
taking $K\in\NN$ large
enough, one can make the right side of
the preceding bound as small as one wish.
Then the first probability on the right side
of (\ref{4CM1}) is small as one wish by
(\ref{3CM1}) and taking $n\in\NN_{+}$ 
large enough.
Therefore (\ref{5CM1}) holds true and
we are left to prove 
(\ref{2CM1}) and (\ref{3CM1}).

Recalling notation $T_{nk}$ given by
 (\ref{6CM1}), for each $n\in\NN_{+}$ and 
 $k\in\NN$, we have 
$$
\frac{1}{n}\EE T_{nk}^2=
\frac{\sigma_{\eta}^2}{n}\sum_{i=1}^n
g^2\Big (\frac{i}{n}\Big )\leq
\sigma_{\eta}^2\|g\|_{\sup}^2.
$$
This proves (\ref{2CM1}).
To prove (\ref{3CM1}) let $k\in\NN_{+}$.
Changing the index $i$ of summation 
it follows that the representation
$$  
T_{nk}-T_{n0}=\sum_{i=1-k}^0\eta_ig\Big (
\frac{i+k}{n}\Big )+\sum_{i=1}^{n-k}
\eta_i\Big [g\Big (\frac{i+k}{n}\Big ) -
g\Big (\frac{i}{n}\Big )\Big ]-
\sum_{i=n-k+1}^n\eta_ig\Big (\frac{i}{n}
\Big )
$$
holds for each integer $n>k$.
Since $g$ is bounded and $k$ is fixed 
the first and the third sum on the right side divided by $\sqrt{n}$ tend to zero 
in probability $\PP$ as $n\to\infty$.
For the second sum divided by $\sqrt{n}$
we apply the H\"older inequality 
$$
\Big |\frac{1}{\sqrt{n}}\sum_{i=1}^{n-k}
\eta_i\Big [g\Big (\frac{i+k}{n}\Big ) -
g\Big (\frac{i}{n}\Big )\Big ]\Big |
\leq \Big (n^{-\frac{p}{2}}\sum_{i=1}^{n-k}
|\eta_i|^p\Big )^{\frac{1}{p}}\Big (
\sum_{i=1}^{n-k}\Big |g\Big (\frac{i+k}{n}
\Big ) -g\Big (\frac{i}{n}\Big )\Big |^q
\Big )^{\frac{1}{q}}
$$ 
with $p\in\RR$ such that $\frac{1}{p}+
\frac{1}{q}\geq 1$.
The telescoping sum representation and
repeated application of Minkowski inequality
 for sums imply that the inequality
$$
\Big (\sum_{i=1}^{n-k}\Big |g\Big (
\frac{i+k}{n}\Big ) -g\Big (\frac{i}{n}
\Big )\Big |^q\Big )^{\frac{1}{q}}
\leq k\|g\|_{(q)}
$$
holds for each integer $n>k$.
Since $1\leq q<2$, then $(2/p)<1$.
Also, since $k$ is fixed and
 $\EE(|\eta_1|^p)^{\frac{2}{p}}=
\sigma_{\eta}^2<\infty$,
by Marcinkiewicz-Zygmund strong law of 
large numbers, we have
$$
\lim_{n\to\infty}n^{-\frac{p}{2}}
\sum_{i=1}^{n-k}|\eta_i|^p=0\quad
\mbox{with probability 1}.
$$
This completes the proof of (\ref{3CM1}).
Theorem \ref{CM1} is proved.
\end{proof}

By definition of Gaussian process $\nu$, 
 for any $g_1,\dots,g_d\in \cW_q$,
$(\nu (g_1),\dots, \nu (g_d))$ have a jointly normal 
distribution with covariance given by
the inner products $\smallint_0^1g_ig_j
\,d\lambda$, $i,j=1, \dots, d$.

\begin{proposition}\label{AR3}
Suppose $(X_i)_{i\in\ZZ}$ is a 
short-memory linear process and $\nu$ is
the isonormal Gaussian processes 
on $\cL_2[0,1]$.
If $g_1,\dots,g_d\in\cW_q[0,1]$ for some
$1\leq q<2$, then
\begin{equation}\label{1AR3}
n^{-1/2}\big (\nu_n(g_1),\dots,\nu_n(g_d)
\big )\wc \sigma_{\eta}A_{\psi}
(\nu (g_1),\dots,\nu (g_d)),\quad\mbox{as
$n\to\infty$.}
\end{equation}
\end{proposition}

\begin{proof} 
Let $d\in\NN_{+}$ and let $g_1,\dots,g_d\in 
\cW_q[0,1]$.
To prove (\ref{1AR3}) we use the
 Cram\'{e}r-Wold device.
Let $a=(a_1,\dots,a_d)\in\RR^d$.
Recalling definition (\ref{nun}) of $\nu_n$
we have
$$
\sum_{h=1}^da_h\nu_n(g_h)=\nu_n\Big (
\sum_{h=1}^da_hg_h\Big )
$$
for each $n\in\NN_{+}$.
Since $\sum_{h=1}^da_hg_h\in\cW_q[0,1]$
by Theorem \ref{CM1} it follows that
$$
n^{-1/2}\nu_n\Big (\sum_{h=1}^da_hg_h\Big )
\wc \sigma_{\eta} A_{\psi}
\nu\Big (\sum_{h=1}^da_hg_h\Big ),\quad
\mbox{as $n\to\infty$.}
$$
Due to linear isometry of $\nu$ the
convergence
$$
n^{-1/2}a{\cdot}\big (\nu_n(g_1),\dots,
\nu_n(g_d)\big )\wc \sigma_{\eta}A_{\psi}
a{\cdot}(\nu (g_1),\dots,\nu (g_d)),\quad
\mbox{as $n\to\infty$.}
$$
holds.
Since $a\in\RR^d$ is arbitrary,
(\ref{1AR3}) holds by the
 Cram\'{e}r-Wold device.
\end{proof}

\section{Asymptotic equicontinuity}\label{Equicontinuity}

Let $(\CLF,\rho)$ be a pseudometric space.
For each $n\in\NN_{+}=\{1,2,\dots\}$,
let $Z_{nk}$, $k\in\ZZ$, be independent
stochastic processes indexed by 
$f\in \CLF$
and defined on the product probability 
space $(\Omega_n,\CLA_n,\PP_n):=
\bigotimes_{k\in\ZZ}(\Omega_{nk},
\CLA_{nk},\PP_{nk})$ so that 
$Z_{nk}(f,\omega)
=Z_{nk}(f,\omega_k)$ for each $\omega
=(\omega_k)_{k\in\ZZ}$ and $f\in\CLF$.
For each $n\in\NN_{+}$ consider a 
stochastic process defined as a series
$$
\sum_{k\in\ZZ}Z_{nk}(f):=
\lim_{m\to +\infty}\sum_{k=-m}^m
Z_{nk}(f),\quad f\in\cF,
$$
provided the convergence holds almost 
surely.
We write $(Z_{nk})\in 
{\cal M} (\Omega_n,\CLA_n,P_n)$ if every one of the functions
 \begin{equation}\label{M1}
\omega\mapsto\sup\Big\{\Big |\sum_{k\in\ZZ}
e_k\big [Z_{nk}(f,\omega)-Z_{nk}(g,\omega)
\big ]\Big | \colon\,f,g\in\CLF,\,\,\rho (f,g)<\delta\Big\}
 \end{equation}
 and
 \begin{equation}\label{M2}
\omega\mapsto\sup\Big\{\Big |\sum_{k\in\ZZ}
e_k\big [Z_{nk}(f,\omega)-Z_{nk}(g,\omega)\big ]^2\Big | \colon\,f,g\in\CLF,\,\, 
\rho (f,g)<\delta\Big\}
 \end{equation}
is measurable for the completion of the 
probability space 
$(\Omega_n,\CLA_n,\PP_n)$,
for every $\delta >0$ and every tuple  $(e_k)_{k\in\ZZ}$
with $e_k\in\{-1,0,1\}$.

The following is Theorem 2.11.1 in \cite{VandW96} adopted to the
convergence of sums of linear processes.

 \begin{theorem}\label{VW}
 Let $(\CLF,\rho)$ be a totally bounded
pseudometric space.
Under the preceding notation assume that 
$(Z_{nk})\in {\cal M} 
(\Omega_n,{\cal A}_n,P_n)$
and there is a subsequence of positive 
integers $(m_n)_{n\in\NN_{+}}$ such that
\begin{equation}\label{4VW}
\lim_{n\to\infty}P_n^*\big(\big\{\big\|
\sum_{k<-m_n}Z_{nk}+\sum_{k>m_n}Z_{nk}
\big\|_{\cF}>\e\big\}\big)=0\quad
 \mbox{for every $\e >0$,}
\end{equation}
  \begin{equation}\label{1VW}
 \lim_{n\to\infty}\sum_{k=-m_n}^{m_n}
 E^{\ast}\|Z_{nk}\|_{{\cF}}^2\id_{\{\|
 Z_{nk}\|_{{\cF}}>\epsilon\}}= 0\quad
 \mbox{for every $\epsilon >0$,}
 \end{equation}
 \begin{equation}\label{2VW}
 \lim_{n\to\infty}\sup_{\rho (f,g)<
 \delta_n}\sum_{k=-m_n}^{m_n}
 E\big [Z_{nk}(f)-Z_{nk}(g)\big ]^2 
 =0\quad\mbox{for every
 $\delta_n\downarrow 0$,}
 \end{equation}
 \begin{equation}\label{3VW}
 \lim_{n\to\infty}\int_0^{\delta_n}
 \sqrt{\log N(x,
 \mathcal{F}, d_n)}\,dx = 0 \ \
 \textrm{in}\ \ P_n^*\ \ \textrm{for every}
 \ \ \delta_n\downarrow 0,
 \end{equation}
where  $d_n$ is a random pseudometric on
$\cF$ defined for each $n\in\NN_{+}$ and
$f,g\in\CLF$ by 
\begin{equation}\label{5VW}
 d_n(f, g):=\Big (\sum_{k=-m_n}^{m_n}
 \big [Z_{nk}(f)-Z_{nk}(g)\big ]^2
 \Big )^{1/2}.
\end{equation}
Then  $Z_n:=\sum_{k\in\ZZ}(Z_{nk}-EZ_{nk})$
 is asymptotically $\rho$-equicontinuous,
  that is for every $\e> 0$,
$$
\lim_{\delta \downarrow 0}\limsup_n 
P_n^*\big(\big\{ \sup\{|Z_n(f) - Z_n(g)|
\colon\,f,g\in\cF,\,\,\rho(f,g)<\delta\} 
> \e\big\}\big)=0.
$$ 
 \end{theorem}

\begin{proof}
Let $(m_n)$ be a subsequence of positive 
integers $(m_n)_{n\in\NN_{+}}$ such that
(\ref{4VW}) holds.
Clearly $(Z_{nk})_{-m_n\leq k\leq m_n}\in 
{\cal M} (\Omega_n,{\cal A}_n,P_n)$.
Using Theorem 2.11.1 in \cite{VandW96}
one can show that
$$
\lim_{\delta \downarrow 0}\limsup_n 
P_n^*\big(\big\{\sup\big\{\big |
\sum_{k=-m_n}^{m_n}[Z_{nk}(f) - Z_{nk}(g)]
\big|\colon\,f,g\in\cF,\,\,
\rho(f,g)<\delta\big\} > \e)=0
$$ 
for each $\e>0$.
For a given $\epsilon >0$ and for each
$n\in\NN_{+}$ we have
\begin{eqnarray*}
\lefteqn{P_n^*\big(\big\{\sup\{|Z_n(f) 
- Z_n(g)|\colon\,f,g\in\cF,\,\,\rho(f,g)
<\delta\} > \e\big\}\big )}\\
&&\leq
P_n^*\big(\big\{ \sup\big\{\big |
\sum_{k=-m_n}^{m_n}[Z_{nk}(f) - Z_{nk}(g)]
\big|\colon\,f,g\in\cF,\,\,\rho(f,g)<\delta
\big\}>\frac{\e}{2}\big\}\big)\\
&&\quad + P_n^*\big(\big\{\big\|
\sum_{k<-m_n}Z_{nk}+
\sum_{k>m_n}Z_{nk}\big \|_{\cF}>
\frac{\e}{4}\big\}\big ).
\end{eqnarray*}
By hypothesis (\ref{4VW}) the conclusion
follows.
\end{proof}

Since a sequence $X_1, X_2,\dots$ is a short-memory linear process,
a sequence of real numbers $(\psi_j)_{j\in\NN}$ is square summable, and so
each series in (\ref{lin-process}) converges almost surely by L\'evy's 
Equivalence Theorem (e.g. Theorem 9.7.1
in Dudley).
Letting $\psi_k:=0$ for each $k<0$ we obtain
the representation
\begin{equation}\label{ext-lin-process}
X_i=\sum_{k=-\infty}^i\psi_{i-k}\eta_k
=\sum_{k\in\ZZ} \psi_{i-k}\eta_k,\quad
i\in\ZZ.
\end{equation}

\begin{lemma}\label{representation}
Suppose $X_1,X_2,\dots$ is a linear process
given by {\rm (\ref{ext-lin-process})},
$f\colon\,[0,1]\to\RR$ and $\nu_n(f)$ 
is the $n$-th $f$-weighted partial sum 
given by {\rm (\ref{nun})}.
For each $n\in\NN_{+}$ and $k\in\ZZ$, let
$$
a_{nk}(f)=\sum_{i=1}^n
\psi_{i-k}f\Big (\frac{i}{n}\Big ),
$$
here $\psi_{i-k}=0$ if $i<k$.
Then for each $n\in \NN_{+}$
\begin{equation}\label{1representation}
\EE\nu_n^2(f)=\sigma_{\eta}^2\sum_{k\in\ZZ}
a_{nk}^2(f)\quad\mbox{and}\quad\nu_n(f)
=\sum_{k\in\ZZ}a_{nk}(f)\eta_k,
\end{equation}
where the random series converges almost
 surely.
\end{lemma}

\begin{proof}
Let $n\in\NN_{+}$.
For each $i,j\in\{1,\dots,n\}$, 
since filter $(\psi_k)_{k\in\NN}$ is
square summable, the series representation
$$
\EE X_iX_j=\sigma_{\eta}^2\sum_{k\in\ZZ}
\psi_{i-k}\psi_{j-k}
$$
converges absolutely.
Thus we have
\begin{eqnarray*}
\EE\nu_n^2(f)&=&\sum_{i,j=1}^n\EE X_iX_j
f\Big (\frac{i}{n}\Big )f\Big (\frac{j}{n}
\Big )\\
&=&\sigma_{\eta}^2
\sum_{k\in\ZZ}\sum_{i,j=1}^n\psi_{i-k}
\psi_{j-k}f\Big (\frac{i}{n}\Big )
f\Big (\frac{j}{n}\Big )
=\sigma_{\eta}^2\sum_{k\in\ZZ}a_{nk}^2(f).
\end{eqnarray*}
and the series on the right side
converges.
This proves the first equality in
(\ref{1representation}).
The second one follows next
\begin{eqnarray*}
\nu_n(f)&=&\sum_{i=1}^n\Big [
\sum_{k\in\ZZ}\psi_{i-k}\eta_k\Big ]
f\Big (\frac{i}{n}\Big )\\
&=&\sum_{k\in\ZZ}\Big [
\sum_{i=1}^n\psi_{i-k}f\Big (\frac{i}{n}
\Big ) \Big ]\eta_k=\sum_{k\in\ZZ}a_{nk}(f)
\eta_k.
\end{eqnarray*}
The series on the right side converges 
almost surely by L\'evy's Equivalence
Theorem (9.7.1 theorem in \cite{Dudley6})
since $(a_{nk}(f))_{k\in\ZZ}$ is square
summable.
\end{proof}

\section{Proof of Theorem \ref{teo2a}}\label{MainResults}

As shown at the end of this section,
Theorem \ref{teo2a} is a simple corollary of the next theorem.
Following \cite[p. 267]{GandN16} we say
that a set of functions $\cF$ satisfies
the \emph{pointwise countable approximation
property} provided there exists a countable
subset $\cF_0\subset\cF$ such that every
$f$ in $\cF$ is a pointwise limit of
functions in $\cF_0$.
Given a probability measure $Q$ on $([0, 1], \cB_{[0, 1]})$, let $\rho_{2,Q}$
be a pseudometric on $\cF$ with values
$$
\rho_{2,Q}(f,g)=\Big(\int_{[0,1]}(f-g)^2dQ\Big )^{1/2},\quad f,g\in\cF
$$

\begin{theorem}\label{teo2} 
Let $X_1, X_2,\dots$ be a short-memory linear process given by {\rm
(\ref{lin-process})} and let $1 \le  q < 2$.
Suppose that a set of functions 
$\cF \subset \cW_q[0, 1]$ is bounded,
satisfies the pointwise countable 
approximation property and
\begin{equation}\label{teo2:condition}
\int^1_0\sup_{Q\in\cQ}\sqrt{\log N(x,\cF,
 \rho_{2,Q})}\, dx < \infty,
\end{equation}
where $\cQ$ is the set of all probability measures on $([0, 1], \cB_{[0, 1]})$.
There exists a version of
the isonormal Gaussian process $\nu$
restricted to $\cF$ with values in a
separable subset of $\ell^{\infty}(\cF)$,
it is measurable for the Borel sets on its
range and {\rm (\ref{1teo2})} holds.
\end{theorem}

Since $\cF\subset\cW_q[0,1]$ with $q\in [1,2)$, the finite dimensional
distributions of $n^{-1/2}\nu_n$ converge in distribution to the corresponding
finite dimensional distributions of $\nu$ by Proposition \ref{AR3}.
By hypothesis (\ref{teo2:condition}), $\cF$
is totally bounded with respect to 
pseudometric $\rho_2$.
Therefore to prove Theorem \ref{teo2}
we have to show that $n^{-1/2}\nu_n$ is
asymptotically equicontinuous with respect
to $\rho_2$.
To this end we use Theorem \ref{VW}.

For each $n\in\NN_{+}$, $k\in\ZZ$ and
$f\colon\,[0,1]\to\RR$, let
\begin{equation}\label{teo2:proof:eq7}
u_{nk}(f):=\frac{1}{\sqrt{n}}\sum_{i=1}^n
\psi_{i-k}f\Big (\frac{i}{n}\Big )
=\frac{a_{nk}(f)}{\sqrt{n}},
\end{equation}
here $\psi_{i-k}=0$ if $i<k$.
By Lemma \ref{representation} we have
useful series representation
\begin{equation}\label{teo2:proof:eq5}
\frac{\nu_n(f)}{\sqrt{n}}
=\sum_{k\in\ZZ}u_{nk}(f)\eta_k,
\end{equation}
We apply Theorem \ref{VW} to the sequence 
of processes
\begin{equation}\label{teo2:proof:eq9}
Z_{nk}=\big\{Z_{nk}(f):=u_{nk}(f)\eta_k
\colon\,f\in \cF\big\},\quad 
k\in\ZZ,\,\,n\in\NN_{+}.  
\end{equation}

\subsection{Measurability}

We can and do assume that 
$(\eta_k)_{k\in\ZZ}$ is defined on the
product probability space
$$
(\Omega,\cA,\PP)=\bigotimes_{k\in\ZZ}
(\Omega_k,\cA_k,\PP_k)
$$
with its joint distribution equal to
the product of distributions of $\eta_k$.
We will show that $(Z_{nk})_{k\in\ZZ}\in 
{\cal M} (\Omega,\cA,\PP)$ using
the fact that $\cF$ satisfies the
pointwise countable approximation
property.

Given a tuple $e=(e_k)_{k\in\ZZ}$ with
$e_k\in\{-1,0,1\}$, for each 
$i\in\{1,\dots,n\}$ and $\omega\in\Omega$,
let
$$
X_i^e(\omega):=\sum_{k\in\ZZ}e_k\psi_{i-k}
\eta_k(\omega).
$$
By (\ref{teo2:proof:eq9}) and
(\ref{teo2:proof:eq7}), for each pair
$f,g\in\cF$, $n\in\NN_{+}$ and
$\omega\in\Omega$, we have
$$
\sum_{k\in\ZZ}e_k\big [Z_{nk}(f,\omega)-
Z_{nk}(g,\omega)\big ]
=\frac{1}{\sqrt{n}}\sum_{i=1}^nX_i^e(\omega
)(f-g)\Big (\frac{i}{n}\Big )=:
T_n^e(f,g,\omega).
$$
For each $\delta >0$, let $\cF^{\delta}
:=\{(f,g)\in\cF\times\cF\colon\,\rho_2
(f,g)<\delta\}$.
Let $\cF_0\subset\cF$ be a countable set
such that every $f\in\cF$ is a pointwise 
limit of functions in $\cF_0$.
Then (\ref{M1}) with $\cF_0$ in place of
$\cF$ is measurable and
$$
\PP^{\ast}\left\{\sup\{|T_n^e(f,g,\cdot)|
\colon\,(f,g)\in\cF^{\delta}\}
\not =
\sup\{|T_n^e(f,g,\cdot)|\colon\,(f,g)
\in\cF_0^{\delta}\}\right\}=0
$$
for each $\delta >0$, each $e=(e_k)_{k\in
\ZZ}$ and each $n\in\NN_{+}$.
Therefore the function (\ref{M1}) is
measurable.

Measurability of (\ref{M2}) follows 
similarly once we show that the series
$$
\omega\mapsto 
\sum_{k\in\ZZ}\big [Z_{nk}(f,\omega)
\big ]^2=\sum_{k\in\ZZ}u_{nk}^2(f)
\eta_k^2(\omega)
$$
converges for each 
$f\in\cF$ and $n\in\NN_{+}$.
But this true due to Lemma
\ref{representation} and due to the
fact that
$$
\EE\left [\sum_{k\in\ZZ}u_{nk}^2(f)\eta_k^2
\right ] = \frac{\sigma_{\eta}^2}{n}
\sum_{k\in\ZZ}a_{nk}^2(f)<\infty.
$$
Therefore $(Z_{nk})_{k\in\ZZ}\in 
{\cal M} (\Omega,\cA,\PP)$.

\subsection{Hypothesis (\ref{4VW})}

By definition, for each $n\in\NN_{+}$
we have $u_{nk}=0$ for each $k>n$.
Therefore $\sum_{k>n}\|u_{nk}\|_{\cF}=0$
for each $n\in\NN_{+}$.
We will choose a subsequence of positive
integers $(m_n)_{n\in\NN_{+}}$ such that
\begin{equation}\label{teo2:proof:eq6}
\lim_{n\to\infty}\sum_{-\infty<k< -m_n}
\|u_{nk}\|_{\cF}=0.
\end{equation}
Let $\FF_{\cF}$ be the function on $[0,1]$
with values
$$
\FF_{\cF}(x):= \sup\{|f(x)|: f \in \cF\},\ \  x\in [0, 1]. 
$$
Since $\cF$ is bounded in $\cW_q[0,1]$,
then $\|\FF_{\cF}\|_{\sup}<\infty$.
By (\ref{teo2:proof:eq7}), for each 
$n\in\NN_{+}$, $k\in\ZZ$ and $f\in\cF$ we have
\begin{equation}\label{teo2:proof:eq8}
|u_{nk}(f)|\leq\frac{1}{\sqrt{n}}
\sum_{i=1}^n|\psi_{i-k}|\Big |f\Big (
\frac{i}{n}\Big
)\Big |\leq\frac{\|\FF_{\cf}\|_{\sup}
}{\sqrt{n}}\sum_{i=1}^n|\psi_{i-k}|.
\end{equation}
Let $0\leq m<M$.
Then
\begin{eqnarray*}
\sum_{-M\leq k\leq -m}\|u_{nk}\|_{\cF}
&\leq&\frac{\|\FF_{\cF}\|_{\sup}}{\sqrt{n}}
\sum_{-M\leq k\leq -m}\sum_{i=1}^n
|\psi_{i-k}|\\
&=&\frac{\|\FF_{\cF}\|_{\sup}}{\sqrt{n}}
\sum_{i=1}^n\sum_{-M\leq k\leq -m}
|\psi_{i-k}|\\
&\leq&\frac{\|\FF_{\cF}\|_{\sup}}{\sqrt{n}}
\sum_{i=1}^n\sum_{j\geq i+m}|\psi_j|
\leq\frac{\|\FF_{\cF}\|_{\sup}}{\sqrt{n}}n
\sum_{j\geq 1+m}|\psi_j|.
\end{eqnarray*}
Now, one can choose a subsequence of
positive integers $(m_n)_{n\in\NN_{+}}$
such that
$\sum_{j\geq 1+m_n}|\psi_j|\leq n^{-1}$
for each $n\in\NN_{+}$.
Hence
$$
\sum_{-\infty<k\leq -m_n}\|u_{nk}\|_{\cF}
\leq \frac{\|\FF_{\cF}\|_{\sup}}{\sqrt{n}}
$$
for each $n\in\NN_{+}$,
and so (\ref{teo2:proof:eq6}) holds.
One can assume that $m_n>n$ for each 
$n\in\NN_{+}$, and so (\ref{4VW}) holds
with the subsequence $(m_n)$.

\subsection{Hypothesis (\ref{1VW})}

To establish hypothesis (\ref{1VW}) 
it is enough to prove that
\begin{equation}\label{teo2:proof:eq1}
U:=\sup_{n\ge 1}\Big (\sum_{k\in\ZZ}
\|u_{nk}\|_{\cF}^2\Big )<\infty.
\end{equation}
Indeed, suppose it is true.
By  (\ref{teo2:proof:eq8}) and assumption
(\ref{filter:1}) we  have
$$
\|u_{nk}\|_{\cF}\leq\frac{\|\FF_{\cf}
\|_{\sup}}{\sqrt{n}}\sum_{j\in\NN}
|\psi_{j}|=:\frac{c}{\sqrt{n}}.
$$
By (\ref{teo2:proof:eq9}), for each 
$m,n\in\NN_{+}$ and $\e>0$ we have
\begin{eqnarray*}
\sum_{k=-m}^mE^{\ast}\|Z_{nk}
\|_{{\cF}}^2\id_{\{\| Z_{nk}\|_{{\cF}}>
\e\}}
&\leq&\sum_{k=-m}^m\|u_{nk}\|_{\cF}^2 
\EE\eta_k^2\id_{\{\| u_{nk}\|_{{\cF}}
|\eta_k|>\e\}}\\
&\leq& U\EE\eta_0^2\id_{\{|\eta_0|>c^{-1}\e
\sqrt{n}\}}.
\end{eqnarray*}
This yields (\ref{1VW}).
We are left to prove
 (\ref{teo2:proof:eq1}).

By (\ref{teo2:proof:eq8}) it is enough
to prove that
\begin{equation}\label{teo2:proof:eq2}
\sup_{n\ge 1}\frac{1}{n}\sum_{k\in \eZ}
\Big(\sum_{i=1}^n |\psi_{i-k}|\Big)^2
<\infty.
\end{equation}
For each $i\in\ZZ$ let
$$
\wt{X}_i=\sum_{k\in\NN}
|\psi_{k}|\eta_{i-k}.
$$
Then
\begin{align*}
\sum_{k\in \eZ}\Big(\sum_{i=1}^n |\psi_{i-k}|\Big)^2&=
E\Big(\sum_{k\in \eZ}\Big(\sum_{i=1}^n|\psi_{i-k}|\Big)\eta_k\Big)^2\\
&=E\Big(\sum_{i=1}^n\sum_{k\in \eZ}|\psi_{i-k}|\eta_k\Big)^2
=E\Big(\sum_{i=1}^n \wt{X}_i\Big)^2.
\end{align*}
Since the linear process $(\wt{X}_i)$ is
covariance stationary, we have
\begin{align*}
E\Big(\sum_{i=1}^n \wt{X}_i\Big)^2&=\sum_{i,j=1}^n E(\wt{X}_i\wt{X}_j)
=n\sum_{j=-(n-1)}^{n-1}\Big (1-
\frac{|j|}{n}\Big )E(\wt{X}_0\wt{X}_j)\\
&\le n\sum_{j=0}^n |E(\wt{X}_j\wt{X}_0)|
=n\sigma^2\sum_{j=0}^n\sum_{k=0}^\infty|\psi_{k+j}|\cdot|\psi_k|\\
&\le \sigma^2\Big(\sum_{k=0}^\infty|\psi_k|\Big)^2 n.
\end{align*}
Due to assumption (\ref{filter:1}),
this completes the proof of (\ref{teo2:proof:eq2}).

\subsection{Hypothesis (\ref{2VW})}

To prove hypotheses (\ref{2VW}) and (\ref{3VW}) 
we use the following representation of the
series (\ref{teo2:proof:eq5}).
For a sequence $(t_k)_{k\in\ZZ}$ of real 
numbers such that 
$\sum_{k\geq 0}\psi_kt_{i-k}$
converges for each $i\in\NN_{+}$, 
the series $\sum_{k\in\ZZ}\psi_{i-k}t_k$
also converges (here $\psi_k=0$ for $k<0$),
 and for each $n\in\NN_{+}$ we have
\begin{eqnarray}
\sum_{k\in\ZZ}\Big [\sum_{i=1}^n\psi_{i-k}f
\Big (\frac{i}{n}\Big )\Big ]t_k
&=&\sum_{i=1}^n\Big [\sum_{k\in\ZZ}
\psi_{i-k}t_k
\Big ]f\Big (\frac{i}{n}\Big )\nonumber\\
&=&\sum_{i=1}^n\Big [\sum_{k=0}^{\infty}
\psi_{k}
t_{i-k}\Big ]f\Big (\frac{i}{n}\Big )
=\sum_{k=0}^{\infty}\psi_k\Big [
\sum_{i=1}^nt_{i-k}f\Big (\frac{i}{n}\Big )
\Big ].\label{teo2:proof:eq4}
\end{eqnarray}

Now to establish hypothesis (\ref{2VW})
recall  ((\ref{teo2:proof:eq7}) and
(\ref{teo2:proof:eq9})) that 
$$
E[Z_{nk}(f)-Z_{nk}(g)]^2
=E[Z_{nk}(f-g)]^2=
\frac{\sigma^2}{n}\Big(\sum_{i=1}^n
\psi_{i-k}(f-g)\Big(
\frac{i}{n}\Big )\Big)^2.
$$
for all $f,g\in \cF$, $n\in\NN_{+}$ and 
$k\in \eZ$.
Let $(r_k)_{k\in\ZZ}$ be a Rademacher
sequence, $h\in\cF$ and $n\in\NN_{+}$. 
By Khinchin-Kahane inequality with the 
constant $K$ we have
\begin{eqnarray*}
\sum_{k\in\ZZ}
\Big(\sum_{i=1}^n\psi_{i-k}h
\Big(\frac{i}{n}\Big )\Big)^2
&=&E\Big (\sum_{k\in\ZZ}
\Big(\sum_{i=1}^n\psi_{i-k}h
\Big(\frac{i}{n}\Big )\Big )r_k\Big )^2\\
&\leq& K^2\Big (E\Big |\sum_{k\in\ZZ}
\Big(\sum_{i=1}^n\psi_{i-k}h
\Big(\frac{i}{n}\Big )\Big )r_k\Big |
\Big )^2.
\end{eqnarray*}
The series on the right side converges and
has representation (\ref{teo2:proof:eq4}) 
with $t_k=r_k(\omega)$.
Therefore
$$
E\Big |\sum_{k\in\ZZ}
\Big(\sum_{i=1}^n\psi_{i-k}h
\Big(\frac{i}{n}\Big )\Big )r_k\Big |
=E\Big |\sum_{k=0}^{\infty}\psi_k
\Big (\sum_{i=1}^nr_{i-k}h
\Big(\frac{i}{n}\Big )\Big )\Big |
$$
$$
\leq\sum_{k=0}^{\infty}|\psi_k|E\Big |
\sum_{i=1}^nr_{i-k}h
\Big(\frac{i}{n}\Big )\Big |
\leq\sum_{k=0}^{\infty}|\psi_k|
\Big (\sum_{i=1}^n
h^2\Big(\frac{i}{n}\Big )\Big )^{1/2}.
$$
Using Minkowski inequality for integrals
and then Minkowski inequality for sums we 
obtain
\begin{align}
\lefteqn{\Big(\frac{1}{n}\sum^n_{i=1}h^2\Big(\frac{i}{n} \Big )\Big)^{1/2}=
\Big(\sum^n_{i=1}\int^{i/n}_{(i-1)/n}
\Big[h\Big (\frac{i}{n}\Big )- h(t) + h(t)
 \Big]^2dt\Big)^{1/2}}\nonumber\\
&\le \Big(\sum^n_{i=1}\Big \{\Big (
\int^{i/n}_{(i-1)/n}\Big[h\Big (\frac{i}{n}
\Big )- h(t)\Big ]^2\,dt\Big )^{1/2}+
\Big (\int_{(i-1)/n}^{i/n}h^2(t)\,
dt\Big )^{1/2}\Big \}^2\Big )^{1/2}
\nonumber\\
&\le\Big(\frac{1}{n}\sum^n_{i=1}\sup\Big\{
\Big [h\Big (\frac{i}{n}\Big )-h(t)\Big ]^2:
 t \in \Big [\frac{i-1}{n},\frac{i}{n}
 \Big ]\Big\}\Big)^{1/2}+\Big(\int^1_0 
 h^2(t) dt\Big)^{1/2}.\label{teo2:proof:eq3}
\end{align}
For each $i\in\{1,\dots,n\}$, we have the bound
$$
\sup\Big\{\Big [h\Big (\frac{i}{n}\Big )-
h(t)\Big ]^2\colon\, t \in \Big [
\frac{i-1}{n},\frac{i}{n} \Big ]\Big\}
\le v_2\Big(h;\Big [\frac{i-1}{n},
\frac{i}{n}\Big ]\Big ).
$$
Summing the bounds over $i$ and continuing
to bound the right side of
 (\ref{teo2:proof:eq3}) it follows that
$$
\Big(\frac{1}{n}\sum^n_{i=1}h^2\Big(
\frac{i}{n} \Big )\Big)^{1/2}
\le n^{-1/2}||h||_{(2)} + \rho_2(h, 0).
$$
Summing up the preceding inequalities and
replacing $h$ by $f-g$ it follows that
$$
\Big (\sum_{k\in\ZZ}E[Z_{nk}(f)
-Z_{nk}(g)]^2\Big)^{1/2}
\le \sigma K\Big(\sum_{k=0}^{\infty}
|\psi_k|\Big)
\Big [\frac{||f-g||_{(2)}}{\sqrt{n}}+
\rho_2(f,g)\Big].
$$
Since $\cF\subset\cW_q[0,1]\subset
\cW_2[0,1]$, this proves hypothesis
 (\ref{2VW}).

\subsection{Hypothesis (\ref{3VW})}

To establish hypothesis (\ref{3VW})
recall the random pseudo-metric
$d_n(f,g)$ defined by (\ref{5VW}).
Since the function $f\mapsto Z_{nk}(f)$
defined by (\ref{teo2:proof:eq9})
is linear, for simplicity, consider instead 
$p_n(f):=d_n(f,0)$ for each $f\in\cF$.
Let $n\in\NN_{+}$, $f\in\cF$ and let
$(r_k)_{k\in\ZZ}$ be a Rademacher
sequence.
By Khinchin-Kahane inequality with the 
constant $K$ again, we have
\begin{equation}\label{teo2:proof:eq10}
p_n^2(f)=\sum_{k\in\ZZ}u^2_{nk}(f)\eta_k^2
=E_r\Big(\sum_{k\in\ZZ}u_{nk}(f)\eta_kr_k
\Big)^2\le K^2\Big(E_r\Big|\sum_{k\in\ZZ}
u_{nk}(f)\eta_kr_k\Big|\Big)^2.
\end{equation}
Now recall notation (\ref{teo2:proof:eq7})
for $u_{nk}(f)$.
Expression (\ref{teo2:proof:eq4})
with $t_k=\eta_k(\omega_1)r_k(\omega_2)$
for the series on the right side
gives equality
$$
\sum_{k\in\ZZ}u_{nk}(f)\eta_kr_k
=\frac{1}{\sqrt{n}}\sum_{k=0}^{\infty}
\psi_k\sum_{i=1}^nf\Big (\frac{i}{n}\Big )
\eta_{i-k}r_{i-k}.
$$
Continuing (\ref{teo2:proof:eq10}) with
this representation we obtain
\begin{align}
p_n(f)&\le\frac{K}{\sqrt{n}}
E_r\Big|\sum_{k=0}^\infty \psi_k 
\sum_{i=1}^nf\Big (\frac{i}{n}\Big ) 
\eta_{i-k}r_{i-k}\Big|
\le \frac{K}{\sqrt{n}}
\sum_{k=0}^\infty |\psi_k| E_r\Big|\sum_{i=1}^n f\Big (\frac{i}{n}\Big )
\eta_{i-k}r_{i-k} 
\Big|\label{teo2:proof:eq11}\\
&\le\frac{K}{\sqrt{n}}
\sum_{k=0}^{\infty}|\psi_k|\Big(
\sum_{i=1}^nf^2\Big (\frac{i}{n}\Big )
\eta_{i-k}^2\Big)^{1/2}
\leq \frac{K}{\sqrt{n}}
\Big (\sum_{k=0}^{\infty}|\psi_k|
\Big )^{1/2}\Big (\sum_{k=0}^{\infty}
|\psi_k|\sum_{i=1}^nf^2\Big (
\frac{i}{n}\Big )\eta_{i-k}^2
\Big)^{1/2}.\nonumber
\end{align}
The last inequality is H\"older's
 inequality.
On $([0,1],\cB)$ define random measures
$\mu_n$ by
$$
\mu_n(B):=\sum_{k=0}^{\infty}|\psi_k|
\frac{1}{n}\sum_{i=1}^n\eta_{i-k}^2
\delta_{i/n}(B),\quad B\in\cB,\,\,n\in
\NN_{+}.
$$
Since $\sigma_{\eta}\not =0$, 
given $\epsilon >0$ one can find 
$\Omega_{\epsilon}\subset\Omega$ and 
$n_{\epsilon}\in\NN$ such that 
$\PP (\Omega_{\epsilon})<\epsilon$ and
$\mu_n([0,1])>0$ for each $\omega\not
\in\Omega_{\epsilon}$ and $n\geq 
n_{\epsilon}$.
Thus without loss of generality we assume
that $\mu_n([0,1])>0$ almost surely.
Then $Q_n:=\mu_n/\mu_n([0,1])$, 
$n\in\NN_{+}$, are random probability 
measures on $([0,1],\cB)$.
For each $n\in\NN_{+}$ let
$$
\xi_n:=K\Big (\sum_{k=0}^{\infty}|\psi_k|
\Big )^{1/2}\sqrt{\mu_n([0,1])}=
K\Big (\sum_{k=0}^{\infty}|\psi_k|\Big 
)^{1/2}\Big (\sum_{k=0}^{\infty}|\psi_k|\frac{1}{n}\sum_{i=1}^n\eta_{i-k}^2
\Big )^{1/2}.
$$
By (\ref{teo2:proof:eq11}) it then follows
that
$$
d_n(f,g)=p_n(f-g)\leq\xi_n\rho_{2,Q_n}
(f,g).
$$
By hypothesis (\ref{teo2:condition}) the set
$\cF$ is totally bounded with respect to
pseudometric $\rho_{2,Q_n}$.
Given $x>0$,
since each $\rho_{2,Q_n}$-ball of radius 
$x/\xi_n$ is contained in a $d_n$-ball 
of radius $x$,  we have 
$$
N(x,\cF,d_n)\leq N(x/\xi_n,\cF,
\rho_{2,Q_n}).
$$
Then by a change of variables it follows that for each $\delta > 0$,
$$
I(\delta):=\int^\delta_0\sqrt{\log 
N(x, \cF, d_n)}\, dx \leq \xi_n
\int^{\delta/\xi_n}_0\sqrt{\log 
N(x, \cF, \rho_{2,Q_n})}\, dx.
$$
For each $\delta >0$, let
$$
J(\delta):=\int^\delta_0\sup_{Q\in\cQ}\sqrt{\log N(x, \cF, \rho_{2,Q})}\, dx.
$$
Let $\e>0$ and let $\delta_n \downarrow 0$. For each $0 < m <M < \infty$ and 
$n \in \NN_{+}$, we have
\begin{equation}\label{teo2:proof:eq12}
\PP(I(\delta_n) > \e) \le 
\PP(MJ(\delta_n/m) > \e)+\PP(\xi_n>M)+ 
\PP(\xi_n <m).
\end{equation}
Taking $m>0$ small enough
the rightmost probability tends to zero 
with $n\to\infty$
since $\liminf_{n\to\infty}\xi_n \geq 
c\sigma_{\eta}>0$ almost surely.
For the next to rightmost probability we
have
$$
\sup_{n\ge 1}\PP(\xi_n>M) \le 
M^{-2}\sup_{n\ge 1} \EE\xi_n^2=
\frac{K^2\sigma_{\eta}^2}{M^2}\Big (
\sum_{k=0}^{\infty}|\psi_k|\Big )^2\to 0 
$$
as $M \to\infty$.
Since $J(\delta_n/m) \to 0$ as $n \to\infty$ by condition (\ref{teo2:condition}) the
first probability on the right side of
(\ref{teo2:proof:eq12}) is zero for
sufficiently large $n$.
It then follows that hypothesis (\ref{3VW}) holds.

Summing up, by Theorem \ref{VW}, 
$n^{-1/2}\nu_n$ is asymptotically 
equicontinuous with respect to $\rho_2$.
By Proposition \ref{AR3} the 
finite-dimensional distributions of
$n^{-1/2}\nu_n$ to finite-dimensional
distributions of $\sigma_{\eta}A_{\psi}\nu$.
Thus by Theorem 3.7.23 in \cite{GandN16}
the conclusion of Theorem \ref{teo2}
is proved.

\vskip10pt

 

\subsection{Proof of Theorem \ref{teo2a}}

The set $\cF_{q,M}$ satisfies the pointwise
countable approximation property as
it is proved in Example 3.7.13 in
\cite[p. 235]{GandN16}.
Clearly the pointwise countable
approximation property holds for a subset
$\cF$ of $\cF_{q,M}$.
Condition (\ref{teo2:condition}) holds
by Theorem 5 in \cite{NandR08}.
Therefore all the hypotheses of Theorem
\ref{teo2} hold true, and its conclusion
also holds true.
The proof of  Corollary \ref{teo2a} is
complete.

\section{Proofs of Theorems \ref{dualitymap} and \ref{teo1}}\label{Pvariation}

We begin with the proof of Theorem
\ref{dualitymap}.
Let $G:=T(L)$,
$(t_i)_{i=0}^m$ be a partition of 
$[a,b]$ and let $b=(b_1,\dots,b_m)\in\RR^m$.
Then $f_b:=\sum_{i=1}^mb_i
\id_{(t_{i-1},t_i]}\in\cW_q[a,b]$ and
$$
\Big |\sum_{i=1}^mb_i[G(t_i)-G(t_{i-1})]
\Big |=\Big |L\Big (\sum_{i=1}^mb_i
\id_{(t_{i-1},t_i]}\Big )\Big |
\leq \|L\|_{\cF_q}\|f_b\|_{[q]}.
$$
Let $\|b\|_q:=(\sum_{i=1}^m|b_i|^q)^{1/q}$.
Then $\|f_b\|_{\sup}=\max_i|b_i|\leq\|b\|_q$
and $\|f_b\|_{(q)}\leq 2\|b\|_q$ due to
Minkowski inequality.
Using extremal H\"older's equality 
we obtain the bound
$$
\Big (\sum_{i=1}^m|G(t_i)-G(t_{i-1})|^p
\Big )^{1/p}=\sup\Big\{\Big |\sum_{i=1}^m
b_i[G(t_i)-G(t_{i-1})]\Big |\colon\,
\|b\|_q\leq 1\Big \}\leq 3\|L\|_{\cF_q}.
$$
Since partition $(t_i)_{i=1}^m$ of $[a,b]$
is arbitrary, it follows that 
$\|G\|_{(q)}\leq 3\|L\|_{\cF_q}$.
Since $\|\id_{[a,(\cdot)]}\|_{[q]}$ is equal to $2$,
we have the bound $\|G\|_{\sup}\leq
2\|L\|_{\cF_q}$ and so (\ref{1dualitymap})
holds.
The proof of Theorem \ref{dualitymap} is
complete.

To prove Theorem \ref{teo1},
for $p\in (2,\infty)$ given as the 
hypothesis, let $q:=(p-1)/p$.
Then $p^{-1}+q^{-1}=1$ and 
$1<q<2$.
By Corollary \ref{teo2a}, the isonormal
Gaussian process $\nu$ restricted to
$\cF_q=\{f\in\cW_q[0,1]\colon\,\|f\|_{[q]}
\leq 1\}$ takes values in a separable
subset of $\ell^{\infty}(\cF_q)$, it is
measurable for Borel sets on its range and
\begin{equation}\label{2teo2a}
(\sigma_{\eta}|A_{\psi}|\sqrt{n})^{-1}
\nu_n\wsc \nu
 \quad\mbox{in $\ell^\infty(\cF_q)$.}
\end{equation}
By the Skorokhod-Dudley-Wichura 
representation theorem (Theorem 3.5.1 in
\cite{Dudley5}), there exist a probability
space $(S,\cS,Q)$ and perfect measurable
functions $g_n\colon\,S\to\Omega$ such that
$Q{\circ}g_n^{-1}=\PP$ on $\cA$ for each
$n\in\NN$ and
\begin{equation}\label{1teo2a}
\lim_{n\to\infty}\|(\sigma_{\eta}|A_{\psi}|
\sqrt{n})^{-1}\nu_n{\circ}g_n-
\nu{\circ}g_0\|_{\cF_q}^{\ast}=0\quad
\mbox{almost surely.}
\end{equation}
Here as for any real-valued function $\phi$ on a probability space,
$\phi^{\ast}$ is its measurable cover which always exists (e.g.,
Theorem 3.2.1 in \cite{Dudley5}).
For each $n\in\NN_{+}$ and  $s\in S$ let
$$
\mu_n(f,s):=\frac{\nu_n(f,g_n(s))}{\sigma_{\eta}|A_{\psi}|\sqrt{n}}, \,\,
f\in\cF_q,\quad\mbox{and}\quad
W_n(t,s):=\mu_n(\id_{[0,t]},s),\,\,
t\in [0,1].
$$
Also, for each $n\in\NN_{+}$, $s\in S$
and $f\in\cF_q$,
$$
|\mu_n(f,s)|\leq\frac{\sum_{i=1}^n|X_i(g_n
(s))|}{\sigma_{\eta}|A_{\psi}|\sqrt{n}}
\|f\|_{\sup}
$$
Hence for each $n\in\NN_{+}$ and $s\in S$,
$\mu_n(\cdot,s)$ is a linear bounded 
functional on $\cW_q[0,1]$.
Given $n,m\in\NN_{+}$ let $L:=\mu_n-\mu_m$.
Then $T(L)(t)=L(\id_{[0,t]})=W_n(t)-W_m(t)$ 
for each $t\in [0,1]$ and
$$
\|W_n-W_m\|_{[p]}\leq 5
\|\mu_n-\mu_m\|_{\cF_q}
$$ 
for each $s\in S$,
by Theorem \ref{dualitymap}.
For any functions $\phi,\xi\colon\,
S\to\RR$, we have $(\phi+\xi)^{\ast}\leq
\phi^{\ast}+\xi^{\ast}$ almost surely
(e.g. Lemma 3.2.2 in \cite{Dudley5}).
By (\ref{1teo2a}) it then follows that
$$
\lim_{m, n\to\infty}\|\mu_n-\mu_m
\|_{\cF_q}^{\ast}=0\quad
\mbox{almost surely.}
$$
Therefore, for each $s\in S$, 
$(W_n(\cdot,s))$ is a Cauchy 
sequence in the Banach space $\cW_p[0,1]$.

For each $s\in S$,  let $W(s):=\{W_t(s)
\colon\,t\in [0,1]\}\in\cW_p[0,1]$ be
a function such that 
$\|W_n(\cdot,s)-W(s)\|_{[p]}\to 0$
almost surely as $n\to\infty$.
For each $t\in [0,1]$, since 
$|W_t-W_n(t)|\leq \|W-W_n\|_{\sup}\to 0$
as $n\to\infty$, $W_t$ is measurable,
and so $W$ is a stochastic process.
For a Borel set $B\in\RR^k$ and $t_1,\dots,
t_k\in [0,1]$, we have
$$
Q\big (\big\{(W_n(t_1),\dots,W_n(t_k))\in B
\big\}\big )=\PP\big (\big\{(\sigma_{\eta}
|A_{\psi}|\sqrt{n})^{-1}(\nu_n(\id_{[0,t_1]}
),\dots,\nu_n(\id_{[0,t_k]})\in B\big\}
\big ).
$$
By (\ref{2teo2a}), the finite dimensional distributions (f.d.d.)
$(\sigma_{\eta}|A_{\psi}|\sqrt{n})^{-1}
(\nu_n(\id_{[0,t_1]}),\dots,
\nu_n(\id_{[0,t_k]})$ converge in distribution as
$n\to\infty$ to the f.d.d.
$(\nu(\id_{[0,t_1]},\dots,\nu (\id_{[0,t_k]}
))$ of the isonormal Gaussian process $\nu$
on $\cL^2([0,1])$.
Also, the f.d.d. $(W_n(t_1),\dots,  
W_n(t_k))$ converge in distribution as $n\to\infty$
to the f.d.d. $(W_{t_1},\dots,W_{t_k})$.
It then follows that $W$ is a Gaussian 
process with the covariance of a Wiener process. 
Since sample paths of $\nu$ are uniformly
continuous with respect to the pseudo-metric
$\rho_2$, $W$ has almost all sample paths
continuous, and so $W$ is a standard Wiener
process on $[0,1]$.

Let $\cC\cW_p^{\ast}[0,1]$ be the set of all
 $f\in \cW_p[0,1]$ such that 
$$
\lim_{\epsilon\downarrow 0}\sup\Big\{
\sum_{k=1}^m|f(t_k)-f(t_{k-1})|^p\colon\,0=t_0<t_1<\cdots<t_m=1, 
\,\,\max_k(t_k-t_{k-1})\leq\epsilon\Big\}=0.
$$
Then $\cC\cW_p^{\ast}[0,1]$ is separable
closed subspace of $\cW_p[0,1]$
(\cite{SVK84}).
Since for each $p'>2$, almos all sample
functions of a Wiener process are of
bounded $p'$-variation on $[0,1]$,
by Lemma 2.14 in \cite[Part II]{DN99},
it follows that almost all sample functions
of $W$ are in $\cC\cW_p^{\ast}[0,1]$.
Therefore $W_n$ converges in law to $W$ in
$\cW_p[0,1]$ by Corollary 3.3.5 in
\cite{Dudley5}).
The proof of Theorem \ref{teo1} is
complete.

\section{Applications}\label{Applications}

In this section we apply the preceding results 
to prove uniform asymptotic normality of least 
squares estimators in parametric regression 
models and to detect change points in trends 
of a short memory linear process.   
Throughout this section again $X_1,X_2,\dots$ is a
short memory linear process given by (\ref{lin-process})
with innovations $(\eta_j)$ and summable filter $(\psi_j)$ such that
 (\ref{filter:1}) holds.

\subsection{Simple regression model}

We start with a simple parametric regression
model $Y_j=\beta Z_{nj}+X_j$, $j=1, \dots, n$, 
where $\beta\in \eR$ is an unknown parameter 
and $Z_{nj}$ are explanatory variables for the 
process $(Y_j)$. 
We assume that $Z_{nj}=f(j/n)$ for
some function $f$ on $[0, 1]$. 
Then the least square estimator of $\beta$ is 
$$
\wh{\beta}_n=\wh{\beta}_n(f):=\Big(\sum_{j=1}^n 
f^2\Big (\frac{j}{n}\Big)\Big)^{-1}
\sum_{j=1}^n Y_j f\Big(\frac{j}{n}\Big).
$$
As a choice of the function $f$ in the 
representation of $Z_{nj}$ is not unique, 
finding an admissible class of functions  
becomes  an important task.   
In response to this question
we present what follows from our main result.

Note that equality
\begin{equation}\label{1regression1}
\sum_{j=1}^n f^2\Big (\frac{j}{n}\Big )
[\wh{\beta}_n(f)-\beta]=\sum_{j=1}^n X_jf
\Big (\frac{j}{n}\Big ).
\end{equation}
holds for each real valued function $f$ on
$[0,1]$ and each $n\in\NN_{+}$.
By Theorem \ref{CM1}, it then follows that
$$
W_n(f):=\frac{1}{\sqrt{n}}\sum_{j=1}^n f^2
\Big (\frac{j}{n}\Big )[\wh{\beta}_n(f)-\beta]
\wc\sigma_{\eta}A_{\psi}\nu (f),\quad
\mbox{as $n\to\infty$,}
$$
for each $f\in\cW_q[0,1]$ with $q\in [1,2)$, 
where $\nu$ is the
isonormal Gaussian process on $\cL_2[0,1]$.

As a straightforward consequence of Corollary
\ref{teo2a} and equality (\ref{1regression1}) 
we obtain a weighted asymptotic normality 
of the estimator $\wh{\beta}_n(f)$ uniformly
over the set of functions $\cF_q=\{f\in
\cW_q[0,1]\colon\,\|f\|_{[q]}\leq 1\}$,
$1\leq q< 2$.

\begin{corollary}\label{regression1} 
Let $1\le q<2$. There exists a version of the 
isonormal Gaussian process $\nu$
restricted to $\cF_q$ with values in a
separable subset of $\ell^{\infty}(\cF_q)$,
it is measurable for the Borel sets on its
range and
$$
W_n\wcs \sigma_{\eta}A_{\psi}\nu\quad\mbox{in 
$\ell^{\infty}(\cF_q)$ as $n\to\infty$}.
$$
\end{corollary}

Next we establish the (unweighted) 
asymptotic normality of $\wh{\beta}_n(f)$
uniformly over a  subset of $\cF_q$.
Since each regulated function is a
Riemann function, we have
\begin{equation}\label{1teo3}
I_n(f^2):=\frac{1}{n}\sum_{j=1}^nf^2
\Big (\frac{j}{n}\Big )\to\int_0^1f^2(x)\,dx
=:I(f^2),\quad\mbox{as $n\to\infty$},
\end{equation}
for any regulated function $f$.
Each function having bounded $p$-variation is
regulated (see e.g. \cite[p. 213]{DN99}).
Therefore by (\ref{1regression1}),
Theorem \ref{CM1} and Slutsky's lemma, 
if $f\in\cW_q[0,1]$ for some $1\leq q<2$ and 
$I(f^2)\not =0$, then
$$
n^{1/2}(\wh{\beta}_n(f)-\beta)\wc  
\sigma_{\eta}A_{\psi}N(0,v^2),
\quad\mbox{as $n\to\infty$},
$$
where $N(0,v^2)$ is Gaussian random variable
with mean zero and variance 
$v^2=(I(f^2))^{-1}$.
For each $\delta >0$ and $q\in [1,2)$, let
\begin{equation}\label{5teo3}
\cF_{q,\delta}:=\{f\in\cW_q[0,1]: 
\|f\|_{[q]}\leq 1,\,\,\,I(f^2)>\delta\}.
\end{equation}

\begin{theorem}\label{teo3} 
Let $q\in [1,2)$ and $\delta \in (0,1)$. 
There exists a version of a
Gaussian process $\nu_1$ indexed by
$\cF_{q,\delta}$ with values in a
separable subset of $\ell^{\infty}
(\cF_{q,\delta})$,
it is measurable for the Borel sets on its
range and
$$
\sqrt{n}(\wh{\beta}_n-\beta)
\wcs \sigma_{\eta}A_{\psi}\nu_1\quad\mbox{in 
$\ell^{\infty}(\cF_{q,\delta})$ as 
$n\to\infty$}.
$$
\end{theorem}

\begin{proof}
Let $\nu$ be the isonormal Gaussian process
from Corollary \ref{teo2a}.
For each $f\in\cF_{q,\delta}$ let $\nu_1(f)
:=\nu(f)/I(f^2)$.
Let $z_1,\dots,z_k\in\RR$ and let $f_1,\dots,
f_k\in\cF_{q,\delta}$.
Since $\nu$ is linear, we have
$$
\sum_{i=1}^k\sum_{j=1}^kz_iz_jE[\nu_1(f_i)
\nu_1(f_j)]=E\nu^2\Big (\sum_{i=1}^kz_i
\frac{f_i}{I(f_i^2)}\Big )=\int_0^1
\Big (\sum_{i=1}^kz_i\frac{f_i}{I(f_i^2)}
\Big )^2d\lambda\geq 0.
$$
Therefore $\{\nu_1(f)\colon\,f\in\cF_{q,\delta}
\}$ is a Gaussian process with mean zero and
variance $E\nu_1^2(f)=(I(f^2))^{-1}$.
Since $\nu$ is linear we can consider $\nu_1$
to be $\nu$ restricted to the set $\{f/I(f^2)
\colon\, f\in\cF_{q,\delta}\}$, which is a
subset of $\{f\in\cW_q[0,1]\colon\, 
\|f\|_{[q]}\leq \delta^{-1}\}$.
Therefore the Gaussian process $\nu_1$ on
$\cF_{q,\delta}$ has values in a
separable subset of $\ell^{\infty}
(\cF_{q,\delta})$ and it is measurable for 
the Borel sets on its range.

Using notation (\ref{1AR3}), (\ref{nun}) and
equality (\ref{1regression1}), for each 
$f\in\cF_{q,\delta}$ and $n\in\NN_{+}$, 
we have
\begin{equation}\label{2teo3}
\sqrt{n}((\wh{\beta}_n(f)-\beta)=
\Big (\frac{1}{I_n(f^2)}-\frac{1}{I(f^2)}
\Big )\frac{\nu_n(f)}{\sqrt{n}}+
\frac{1}{\sqrt{n}}\nu_n\Big (\frac{f}{I(f^2)}
\Big ).
\end{equation}
We show that the second term on the right side
converges ir outer distribution to 
$\sigma_{\eta}A_{\psi}\nu_1$.
Let $C(f):=1/I(f^2)$ for each 
$f\in\cF_{q,\delta}$ and let $g\colon\,
\ell^{\infty}(\cF_{q,\delta})\to
\ell^{\infty}(\cF_{q,\delta})$ be a multiplication function with values $g(F):=CF$
for each $F\in\ell^{\infty}(\cF_{q,\delta})$.
One can check that $g$ is bounded and continuous.
Let $h\colon\,\ell^{\infty}(\cF_{q,\delta})
\to\RR$ be a bounded and continuous function.
Then the composition $h{\circ}g$ is also
bounded and continuous function from
$\ell^{\infty}(\cF_{q,\delta})$ to $\RR$.
Thus by (\ref{1teo2}) it follows that
$$
E^{\ast}h\big (n^{-1/2}C\nu_n\big )
=E^{\ast}h{\circ}g(n^{-1/2}\nu_n)\to
Eh{\circ}g(\nu)=Eh(\sigma_{\eta}A_{\psi}\nu_1)
$$
as $n\to\infty$.
Thus the second term on the right side of
(\ref{2teo3}) converges in outer distribution
to $\sigma_{\eta}A_{\psi}\nu_1$.
Finally, we will
show that the first term on the right side of
(\ref{2teo3}) converges to zero as $n\to\infty$
uniformly in $f$ and in outer probability 
(1.9.1 Definition in \cite{VandW96}).

Let $f\in\cF_{q,\delta}$ and $n\in\NN_{+}$.
Then
\begin{eqnarray*}
\lefteqn{I_n(f^2)-I(f^2)=\sum_{j=1}^n
\int_{(j-1)/n}^{j/n}\Big(f^2\Big(\frac{j}{n}
\Big)-f^2(x)\Big)\,dx}\\
&=&2\sum_{j=1}^n\int_{(j-1)/n}^{j/n}f(x)\Big(
f\Big(\frac{j}{n}\Big )-f(x)\Big)\,dx +
\sum_{j=1}^n\int_{(j-1)/n}^{j/n}\Big(f\Big(
\frac{j}{n}\Big )-f(x)\Big)^2\,dx.
\end{eqnarray*}
Since $v_2(f;[a,c])+v_2(f;[c,b])\leq 
v_2(f;[a,b])$ for any $0\leq a<c<b\leq 1$, it
follows that
$$
\sum_{j=1}^n\int_{(j-1)/n}^{j/n}\Big(f\Big(
\frac{j}{n}\Big )-f(x)\Big)^2\,dx\leq
\frac{v_2(f)}{n}.
$$
Then using H\"older's inequality, we obtain
the bound
$$
K_n(f):=
|I_n(f^2)-I(f^2)|\leq 2(I(f^2))^{1/2}
\Big (\frac{v_2(f)}{n}\Big )^{1/2}+
\frac{v_2(f)}{n}.
$$
For each $f\in\cF_{q,\delta}$,
$I(f^2)\leq\|f\|_{\sup}^2\leq 1$ and
$v_2(f)\leq \|f\|_{[q]}^2
\leq 1$ since $q<2$, and so 
$K_n(f)<3/\sqrt{n}$.
We observe that for each $f\in\cF_{q,\delta}$
and $n\in\NN_{+}$
$$
D_n(f):=\Big |\frac{1}{I_n(f^2)}-
\frac{1}{I(f^2)}\Big |\leq\frac{K_n(f)}{
I^2(f^2)(1-I^{-1}(f^2)K_n(f))}
$$
provided $I^{-1}(f^2)K_n(f)<1$.
Therefore for each $n>36/\delta^2$ we have
the bound
\begin{equation}\label{3teo3}
\|D_n\|_{\cF_{q,\delta}}
<\frac{6}{\sqrt{n}\delta^2}.
\end{equation}
To bound the first term on the right side of
(\ref{2teo3}) note that for each $\epsilon >0$
and for each $A\in\RR$
$$
\PP^{\ast}\Big (\Big\{\sup_{f\in\cF_{q,\delta}}
\Big |\frac{1}{I_n(f^2)}-\frac{1}{I(f^2)}
\Big |\frac{|\nu_n(f)|}{\sqrt{n}}>\epsilon
\Big )\Big )
\leq\PP^{\ast}(\{\|D_n\|_{\cF_{q,\delta}}
\|\nu_n\|_{\cF_{q,\delta}}>\epsilon\sqrt{n}\})
$$
\begin{equation}\label{4teo3}
\leq\PP(\{\|D_n\|_{\cF_{q,\delta}}
>\epsilon A\})+\PP^{\ast}(\{\|\nu_n\|_{\cF_{q,
\delta}}>A\sqrt{n}\}).
\end{equation}
By Corollary \ref{teo2a} the isonormal
Gaussian process $\nu$ restricted to 
$\cF_{q,\delta}$ has values in separable
subset of $\ell^{\infty}(\cF_{q,\delta})$
and is measurable for the Borel sets on its 
range.
Therefore its law $\PP{\circ}\nu^{-1}$ is
tight.
By lemma 1.3.8 in \cite{VandW96}, $n^{-1/2}
\nu_n$ is asymptotically tight.
It follows that
for each $\epsilon >0$ there is $A\in\RR$
such that
$$
\limsup_{n\to\infty}
\PP^{\ast}(\{\|\nu_n\|_{\cF_{q,
\delta}}>A\sqrt{n}\})<\epsilon
$$
By (\ref{3teo3}) the first probability on
the right side of (\ref{4teo3}) is zero for
all sufficiently large $n$.
The first term on the right side of
(\ref{2teo3}) converges to zero as $n\to\infty$
uniformly in $f$ and in outer probability.
Therefore the right side of (\ref{2teo3})
converges in outer distribution
to $\sigma_{\eta}A_{\psi}\nu_1$, and so does
the left side of (\ref{2teo3})
by lemma 1.10.2 in \cite{VandW96}.
The proof is complete. 
\end{proof}

For each real valued function $f$ on $[0,1]$
and for each $n\in\NN_{+}$ let
$$
Q_n(f):=\Big [\sum_{j=1}^n f^2\Big (\frac{j}{n}\Big )\Big ]^{1/2}[\wh{\beta}_n(f)-\beta].
$$
Recalling definition (\ref{5teo3}) of
the class of functions $\cF_{q,\delta}$ we
have the following result.

\begin{theorem}\label{teo4} 
Let $q\in [1,2)$ and $\delta \in 
(0,1)$. There exists a version of a
Gaussian process $\nu_2$ indexed by
$\cF_{q,\delta}$ with values in a
separable subset of $\ell^{\infty}
(\cF_{q,\delta})$,
it is measurable for the Borel sets on its
range and
$$
Q_n \wcs \sigma_{\eta}A_{\psi}\nu_2\quad
\mbox{in $\ell^{\infty}(\cF_{q,\delta})$ 
as $n\to\infty$}.
$$
\end{theorem}

\begin{proof}
The proof is similar to the one of Theorem
\ref{teo3}. 
Indeed, using equality (\ref{1regression1}),
notation (\ref{1teo3}) and (\ref{nun})
we have representation
\begin{eqnarray*}
\lefteqn{Q_n(f)=\frac{1}{(I_n(f^2))^{1/2}}
\frac{\nu_n(f)}{\sqrt{n}}}\\
&=&\left [\frac{1}{(I_n(f^2))^{1/2}}
-\frac{1}{(I(f^2))^{1/2}}\right ]
\frac {\nu_n(f)}{\sqrt{n}}+
\frac{1}{\sqrt{n}}\nu_n\Big (\frac{f}{(I(f^2
)^{1/2}}\Big )
\end{eqnarray*}
for each $f\in\cF_{q,\delta}$ and each
$n\in\NN_{+}$.
Note that
$
|\sqrt{u}-\sqrt{v}|\leq\sqrt{|u-v|}
$
for any $u\geq 0$ and $v\geq 0$.
Therefore the above term in the square brackets
approaches zero as $n\to\infty$ uniformly for 
$f\in\cF_{q,\delta}$ due to the bound
(\ref{3teo3}) given in the proof of Theorem
\ref{teo3}.
The conclusion of Theorem \ref{teo4} then 
follows as in the preceding proof.
\end{proof}

\subsection{Multiple change point model}

Consider a time-series model 
$$
Y_{nj}=\mu_{nj}+X_j,\ \  j=1, \dots, n.
$$
which is subject to unknown multiple 
change points
$(\tau_1^*,\dots,\tau_{d^*-1}^*)$.
We wish to test the null hypothesis
$$
H_0:\ \  \mu_{n1}=\cdots=\mu_{nn}=0
$$
against the multiple change alternative model
\begin{equation}\label{HA}
H_A: \ \ \mu_{nj}= \sum_{k=1}^{d^*}\beta_k\bm{1}_{I^*_k}(j/n), \ \ j=1, \dots, n,
\end{equation}
with  $d^*\in\NN_{+}$, $\beta_1,\dots,\beta_{d^*}$, 
and $I^*_k=(\tau^*_{k-1}, \tau^*_k]$, $0=\tau^*_0<\tau^*_1<\cdots<\tau^*_d=1$ being unknown parameters.

There is an enormous amount of literature where
change detection problems have been studied. 
The books by Basseville and 
Nikiforov \cite{BN},   Cs\"orgo  and  Horv\'ath  
\cite{CH},  Brodsky and Darkhovskay \cite{BD},  
Chen, Gupta \cite{ChG} introduce basics on 
various methods. 
We suggest a testing procedure based on uniform 
asymptotic normality of partial sum processes
obtained in the present paper. 

For each $d\in\NN_{+}$, let $\cT_d$ be a set
of all partitions $(\tau_k)$ of the interval 
$[0, 1]$ such that
$0=\tau_0<\tau_1<\cdots<\tau_d=1$. 
For a partition $\tau=(\tau_k)\in \cT_d$
one fits the regression model 
\begin{equation}\label{model2a}
Y_{nj}=\sum_{k=1}^{d}\beta_k\bm{1}_{I_k}(j/n)+X_j,
\ \ j=1, \dots, n,
\end{equation}
where $I_k=(\tau_{k-1}, \tau_k]$.
The parameter
$\bm{\beta}=[\beta_1, \dots, \beta_{d}]'$ 
is obtained by the least square estimator
$$
\wh{\bm{\beta}}=\wh{\bm{\beta}}(\bm{\tau})=\bm{Q}_n^{-1}\Big[\sum_{j=1}^nY_{nj}\bm{1}_{I_1}(j/n), \dots, \sum_{j=1}^nY_{nj}\bm{1}_{I_d}(j/n)\Big]',
$$
where
$$
\bm{Q}_n=\diag\{[\tau_kn]-[\tau_{k-1}n], k=1, \dots, d\}. 
$$
For $p\geq 1$ let $\|x\|_p$ be the 
$\ell_p$-norm of a vector $x\in\eR^d$. 
Then
$$
\|\bm{Q}_n\wh{\bm{\beta}}\|_p
=\Big(\sum_{k=1}^d\Big|\sum_{j=1}^nY_{nj}\bm{1}_{I_k}(j/n)\Big|^p\Big)^{1/p}.
$$
Let
$$
T_n=T_n(Y_{n1}, \dots, Y_{nn}):=
\sup_{d\in\NN_{+}}\sup_{(\tau_k)\in \cT_d}
\|\bm{Q}_n\wh{\bm{\beta}}\|_p.
$$
Under the null hypothesis $H_0$,
by (\ref{teo1b}), the statistic 
$T_n=\|S_n\|_{(p)},$
where $S_n$ is the
partial sum process of a linear process 
$(X_k)$.
The following fact is based on Corollary
\ref{teo1a}.

\begin{theorem} Let $p>2$. Under the null hypothesis $H_0$ it holds
$$
n^{-1/2}A_{\psi}^{-1}\sigma_{\eta}^{-1}T_n
\wc ||W||_{(p)},\quad\mbox{as $n\to\infty$}.
$$
\end{theorem}

Under the alternative $H_A$ we have 
$n^{-1/2}T_n\pc \infty$ provided
$$
\sqrt{n}\Big(\sum_{k=1}^{d^*}(\tau_k^*-\tau^*_{k-1})^p|\beta_k|^p\Big)^{1/p}\to \infty\quad\mbox{as $n\to\infty$.}
$$
Calculations of $p$-variation of piecewise functions are available in $R$ environment package under the
name \texttt{pvar} developed by Butkus and Norvai\v sa \cite{BandN}.  
Simulation analysis of the statistic $T_n$ is out of the scope of the paper and will appear elsewhere.

\end{document}